\documentclass[fontsize=10.8pt,paper=letter]{scrartcl}
\usepackage[vmarginratio={2:3},heightrounded]{geometry}

\usepackage[cmintegrals,cmbraces]{newtxmath}
\usepackage{ebgaramond-maths}
\usepackage[T1]{fontenc}

\usepackage[usenames]{color}
\usepackage{amsmath}
\usepackage{amsthm}
\usepackage{verbatim}
\usepackage{enumerate}
\usepackage{caption}
\usepackage{subcaption}
\usepackage{booktabs}
\usepackage{authblk}

\usepackage{tikz-cd}

\usepackage{verbatimbox}

\usepackage{xpatch}
\xpatchcmd{\maketitle}{\null\vfill}{}{}

\usepackage{algorithm}
\usepackage{algpseudocode}

\usepackage{paralist} 

  \pltopsep=\medskipamount
\usepackage{cite}
\usepackage{hyperref}
\definecolor{PineGreen}{rgb}{0.0,0.47,0.44}
\definecolor{MidnightBlue}{rgb}{0.1,0.1,0.44}
\definecolor{magenta}{rgb}{1.0,0.0,1.0}
\hypersetup{
  colorlinks=true,
  linkcolor=MidnightBlue,
  citecolor=PineGreen,
  filecolor=magenta,
  urlcolor=MidnightBlue
}

\numberwithin{equation}{section}

\theoremstyle{plain}
\newtheorem{question}[equation]{Question}

\theoremstyle{definition}  %

\newtheorem{theorem}[equation]{Theorem}
\newtheorem*{theorem*}{Theorem}

\newtheorem{lemma}[equation]{Lemma}
\newtheorem{corollary}[equation]{Corollary}
\newtheorem{proposition}[equation]{Proposition}

\theoremstyle{definition}
\newtheorem*{definition*}{Definition}
\newtheorem{definition}[equation]{Definition}

\newtheorem{remark}[equation]{Remark}
\newtheorem{example}[equation]{Example}

\newtheorem*{notation*}{Notation}

\renewcommand{\tt}{\texttt}

\newcommand{\Mod}[1]{\ (\mathrm{mod}\ #1)}

\DeclareMathOperator{\im}{Im}
\DeclareMathOperator{\zim}{\overline{\im}}
\DeclareMathOperator{\blow}{Bl}

\newcommand{\ratto}{\dashrightarrow}

\def\a{\mathbb{A}}

\def\kk{\mathbb{K}}

\def\p{\mathbb{P}}
\def\RR{\mathbb{R}}

\DeclareMathOperator{\Hom}{Hom}

\DeclareMathOperator{\depth}{depth}

\newcommand{\inv}{^{-1}}
\newcommand{\Cc}{\mathbb{C}}
\newcommand{\Zz}{\mathbb{Z}}

\newcommand{\cg}{\mathcal{G}}
\newcommand{\ct}{\mathcal{T}}
\newcommand{\xa}{\mathfrak{a}}
\newcommand{\xb}{\mathfrak{b}}
\newcommand{\xr}{\mathfrak{r}}
\newcommand{\xn}{\mathfrak{n}}
\newcommand{\xc}{\mathfrak{c}}
\newcommand{\xM}{\mathfrak{M}}

\newcommand{\bs}[1]{\mathcal{B}_{#1}}

\DeclareMathOperator{\tr}{tr}
\DeclareMathOperator{\imps}{IMPS}
\DeclareMathOperator{\mps}{MPS}
\newcommand{\Par}{\psi_{r,k,q}}
\newcommand{\Parrr}{\psi_{2,2,3}}
\newcommand{\Parr}{\psi_{2,4,3}}
\DeclareMathOperator{\fl}{fl}

\makeatletter
\g@addto@macro{\normalsize}{%
    \setlength{\abovedisplayskip}{4pt}
    \setlength{\abovedisplayshortskip}{2pt}
    \setlength{\belowdisplayskip}{4pt}
    \setlength{\belowdisplayshortskip}{2pt}}
\makeatother

\begin{document}

\title{Computing images of polynomial maps}

\author[1]{Corey Harris}
\author[1,2]{Mateusz Micha{\l}ek\thanks{MM was supported by Polish National Science Center project
2013/08/A/ST1/00804 affiliated at the University of Warsaw.}}
\author[1]{Emre Can Sert\"oz}
\affil[1]{Max Planck Institute for Mathematics in the Sciences}
\affil[2]{Institute of Mathematics of the Polish Academy of Sciences}
\date{}                     
\setcounter{Maxaffil}{0}
\renewcommand\Affilfont{\itshape\small}





\maketitle
\vspace{-0.55in}
\begin{abstract}
The image of a polynomial map is a constructible set.
While computing its closure is standard in computer algebra systems, a procedure for computing the constructible set itself is not. We provide a new algorithm, based on algebro-geometric techniques, addressing this problem.
We also apply these methods to answer a question of W.~Hackbusch on the non-closedness of site-independent cyclic matrix product states for infinitely many parameters.
\end{abstract}
\section{Introduction}\label{sec_intro}

Determining the image of a polynomial map is of fundamental importance in numerous disciplines of mathematics.
In particular, this problem comes up in dealing with parametrizations of (unirational) varieties, a situation which arises frequently in theory and in application, for instance in low-rank tensor approximation.

Given a projective variety $X \subset \p^n_\Cc$, we compute the image of a polynomial map $f: X \ratto \p^m_\Cc$.
This setting easily extends to rational maps from affine varieties to affine spaces---see Section \ref{subsec:affrat}.

Our primary goal is to develop an algorithm to compute this image. We have two design principles regarding the output: first, it should give immediate insight to a human, and second, a computer using our output should be able to determine instantly if a point in the codomain belongs to the image. Let us emphasize here that the output we produce will make it clear at first sight whether or not the image is closed.

We begin with a simple example. Consider the Cremona transformation $f: \p^2 \ratto \p^2$ defined by $[x_0,x_1,x_1] \mapsto [x_1x_2,x_0x_2,x_0x_1]$. For more complicated examples see Section \ref{sec:examples}.

Let us write the image of $f$ as a constructible set $V_0\setminus(V_1 \setminus V_2)$ where $V_0 = \p^2$, $V_1=Z(y_0y_1y_2)$ and $V_2 = Z(y_0y_1,y_0y_2,y_1y_2)$.
Here we represent closed algebraic sets as the zeros of an ideal, written $Z(I)$.
It is more convenient however to decompose the $V_i$'s into their irreducible components and store the containment relations in the form of a graph.

\begin{figure}[h]
\begin{subfigure}[c]{.5\textwidth}
  \centering
  \includegraphics{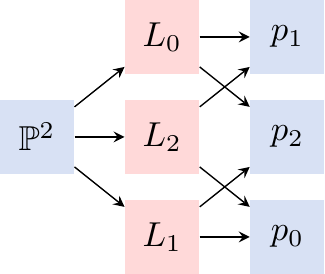}
  \caption{Constructible graph}
  \label{fig:CremonaIntroGraph}
\end{subfigure}%
\begin{subfigure}[c]{.5\textwidth}
  \centering
  \includegraphics{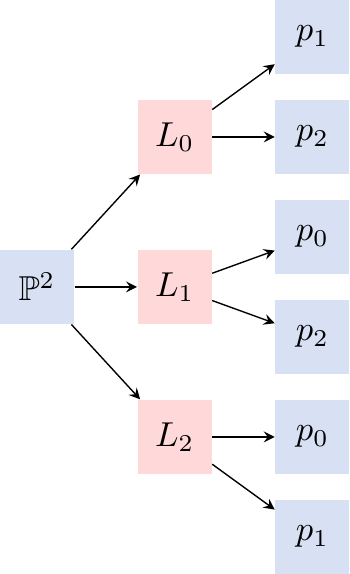}
  \caption{Constructible tree}
  \label{fig:CremonaIntroTree}
\end{subfigure}
\caption{Cremona transformation}
\label{fig:CremonaGraphs}
\end{figure}

Then $V_1 = L_0 \cup L_1 \cup L_2$ and $V_2=p_0\cup p_1 \cup p_2$
where the $L_i$'s (resp.~$p_i$'s) are the three lines (resp.~points) in $\p^2$ defined by the vanishing of coordinates.
The image of the Cremona transformation can now be presented as a graph as in Figure \ref{fig:CremonaIntroGraph}.

In our implementation the image is represented in the form of a tree. For the Cremona transformation it is depicted in Figure \ref{fig:CremonaIntroTree}, meanwhile the output of our implementation is presented in Figure \ref{fig:verbCremonaM2Output}.

\begin{verbbox}
   (2) ideal()
 - (1) |====ideal y2
 + (0) |    |====ideal(y2,y0)
 + (0) |    |====ideal(y2,y1)
 - (1) |====ideal y1
 + (0) |    |====ideal(y1,y0)
 + (0) |    |====ideal(y2,y1)
 - (1) |====ideal y0
 + (0) |    |====ideal(y2,y0)
 + (0) |    |====ideal(y1,y0)
\end{verbbox}
\begin{figure}[ht]
  \centering
  \theverbbox
  \caption{The output of our \tt{TotalImage} for the Cremona transformation}
  \label{fig:verbCremonaM2Output}
\end{figure}

Standard methods exist for determining the \emph{closure} of the image. They rely on Gr\"obner basis computations and are implemented in any general mathematical software, cf.~\S 3.3 \cite{cox1992ideals}.  As far as we are aware, however, the only software which computes the image of a polynomial map is \tt{PolynomialMapImage} in the Maple\texttrademark{} module \tt{RegularChains} \cite{chen2008constructiblesettools, Maple10}. This program uses \emph{triangular decompositions}---a technique well-developed in algorithmics \cite{wen1987zero}, but which is not a part of the canon of algebraic geometry.

Our algorithm  relies on a central technique in algebraic geometry: resolving a rational map through blow-ups. Our implementation of the algorithm is called \tt{TotalImage}\footnote{Available at: \url{https://github.com/coreysharris/TotalImage}}.
It compares favorably to \tt{PolynomialMapImage} in our tests.
See Section \ref{sub:comparison} for a detailed comparison.

We also demonstrate how one can make theoretical use of the idea behind this algorithm to prove that an image is not closed without computing the entire image. In the process, we prove that the set of tensors that admit a \emph{site-independent matrix product state} ($\imps$) representations with fixed rank is not closed (Theorem \ref{tw:main}). This answers a question posed by W.~Hackbusch.
A cousin problem of deciding whether the set of tensors that admit a \emph{matrix product state} ($\mps$) representation form a closed set, posed by L.~Grasedyck, was settled in \cite{landsburg2012geometry}.
\enlargethispage{2\baselineskip}

We will now describe three domains of application in which the determination of the image of a map plays a crucial role.

\subsection{Physics}  \label{sec:phys}

Tensors play a prominent role in physics, for instance in the representation of quantum states. An issue is that  relevant tensors often appear in spaces of huge dimension, making them practically impossible to work with directly.

A way around this problem is to find compact representations of a tensor, such as low-rank presentations (also known as the canonical polyadic decomposition) or tensor networks \cite{hackbusch2012tensor}. In practice, one often gives an algebraic parametrization of a family of well-behaving tensors, such as those admitting compact representations. It is of great concern from the point of view of numerical mathematics to decide whether the image of such a parametrization map is closed.

In other words, one wishes to know in advance whether a sequence of well-behaving tensors $T_n$, approximating an arbitrary tensor $T$, will converge to a good approximation $T_{\infty}$ within the set of well-behaving tensors. For example, for a specified $r$ and a real tensor $T$, there may be no best-possible real-rank $r$ approximation of $T$. In fact, this happens with \emph{positive probability} in the choice of $T$ \cite{de2008tensor}. The complex case, where such phenomena do not take place, along with examples when best rank approximations do not exist, is discussed in \cite{qi2017complex}.

\subsection{Statistics} \label{sec:stats}
A statistical model is a parametric family of probability distributions.
A large class of statistical models are parametrized by algebraic maps \cite{drton2007algebraic, sturmfels2005toric, 74228}.

The primary question about a statistical model is if a given, i.e.~observed, probability distribution fits the model. To attack this question, one wishes to describe the real image of the parametrization corresponding to the model within the space of all probability distributions.

In this paper we only deal with the \emph{complex} image of algebraic maps. However, the complex image, being larger, often gives a good first test for the fitness of a statistical model.

\subsection{Computational Sciences} \label{sec:comp_sci}

Tensors represent multi-linear maps. Good representations of a tensor, for instance its rank decomposition, yield algorithms of lower complexity \cite{landsberg2012tensors, landsberg2017geometry}.

A famous example demonstrating this relationship is matrix multiplication. The multiplication of two $n\times n$ matrices is a bilinear operation and thus is represented by a $3$-dimensional tensor. The complexity of the optimal algorithm for multiplying matrices is known to be governed by the rank (or border rank) of the associated tensor, see \cite{landsberg2012tensors,landsberg2017geometry}. (Let us point out that it is not known in general if the rank and the border rank of the matrix multiplication tensor coincide.)

Computing the tensor rank (as well as determining if the tensor rank equals the border rank) of a given tensor $T$ is equivalent to the problem of deciding whether $T$ belongs to the image of an algebraic map (or its closure).\\

We start by presenting the preliminaries in Section \ref{sec:Prelim}.
In Section \ref{sec:Im}, we present our algorithm for computation of images.
In Section \ref{sec:Hack}, we answer the question of W.~Hackbusch proving that $\imps$ tensors do not form a closed set in general. In Section \ref{sec:examples} we present in detail two explicit examples inspired by statistics and physics. Some of the proofs and remarks are postponed to the Appendix.

\section*{Acknowledgments}
We thank Wolfgang Hackbusch for posing the question which motivated this work, and for the stimulating discussions. We are grateful to Bernd Sturmfels and Michael Joswig for many suggestions and encouraging remarks.
\section{Preliminaries}\label{sec:Prelim}

A map $f: \Cc^n \to \Cc^m$ defined by $x \mapsto (f_1(x),\dots,f_m(x))$ where the $f_i$ are polynomials in the coordinates of $x=(x_1,\dots,x_n)$ is called a \emph{polynomial map}. If the $f_i$ are given as the quotient of two polynomials, then $f$ is called a \emph{map of rational functions}. Note that if the $f_i$ are not polynomial, the map $f$ is not well defined everywhere in the domain and we use the notation $f: \Cc^n \ratto \Cc^m$ to allow for this possibility.

The goal of this paper is to compute the image of a map of rational functions. There is another case of interest however, which turns out to generalize the one above while providing a more advantageous perspective.

Consider a map $f: \p^n \ratto \p^{m}$ defined by $ [x] \mapsto [f_0(x),\dots,f_m(x)]$ where the $f_i$ are rational functions. This makes sense only when the $f_i$ are homogeneous of the same degree.

When each $f_i$ is a polynomial we may emphasize this fact by referring to $f$ as a \emph{polynomial map}. Note that even when $f$ is a polynomial map, $f$ need not be well-defined on the entire domain and we will use the notation $\p^n \ratto \p^{m} $ to highlight this fact.

\subsection{From affine rational to projective polynomial maps}\label{subsec:affrat}
Although one can extend a map of rational functions $f:\Cc^n \ratto \Cc^m$ to a polynomial map $\p^n \ratto \p^m$, the standard way to do this would change the image. The trick below allows one to perform this extension \emph{without changing the image}.

Let $\iota: \Cc^m \hookrightarrow \p^m$ be defined by $(x_1,\dots,x_m) \mapsto [1,x_1,\dots,x_m]$.
The composition $\iota \circ f$ extends to $\p^n \ratto \p^m$ as
\[
 [x_0,\dots,x_n] \mapsto [1,x_0^{-\deg f_1}f_1,\dots,x_0^{-\deg f_m}f_m].
\]
This map is undefined wherever the previous map was undefined and additionally at the hyperplane at infinity (unless the map is constant).
In particular, it has the same image.

Further, we can convert any rational map $f: \p^n \ratto \p^m$, defined by rational functions $f_i = \frac{g_i}{h_i}$, to a polynomial map without changing the image. Set $H := \prod_{i=0}^m h_i$ and define the polynomial map
\[f' = [H^2 f_0, H^2 f_1, \dots ,H^2 f_m].\]
The image of $f$ and $f'$ coincide, while $f'$ is polynomial.

For these reasons, in the rest of this paper we will concentrate on polynomial maps $\p^n \ratto \p^m$ and their restrictions to varieties $X$ in $\p^n$.
\enlargethispage{\baselineskip}

\section{Image of a variety}\label{sec:Im}

  In this section we describe our algorithm for computing the image of a polynomial map defined on a projective variety.  Let us emphasize that we work over the complex numbers and point to the references \cite{cox1992ideals, shafarevich1994basic} for the basic facts we will be using from algebraic geometry.

  \subsection*{Constructible sets}
  Our starting point is Chevalley's theorem on constructible sets.
  \begin{theorem}[Chevalley]
    Let $f: \mathbb{P}^n \dashrightarrow \mathbb{P}^k$ be a rational map and $V \subset \mathbb{P}^n$ a variety.  If $B$ is the base locus of $f$, then $f(V \backslash B)$ is a constructible set.
  \end{theorem}
  In other words, the image can be described by a finite sequence of algebraic sets $(Z_0,\dots,Z_k)$ such that:
    \begin{itemize}
      \item $Z_0 \supsetneq Z_1 \supsetneq \dots \supsetneq Z_k $,
      \item $\im(f) = Z_0 \setminus (Z_1 \setminus (Z_2 \setminus ( \dots (Z_{k-1}\setminus Z_k)\dots)))$.
    \end{itemize}
    The last of these conditions is often written in the form
    \[
      \im(f) = Z_0 - Z_1 + Z_2 - \dots + (-1)^k Z_k.
    \]
    Here we note that the subtraction and addition operations on sets do not commute.

    \begin{definition}\label{def:constr}
  For a constructible set $C$ a representation $C = V_0 - V_1 + \dots + (-1)^\ell V_\ell$ will be called \emph{canonical} if the following properties hold:
  \begin{align*}
    V_{2k} &= \overline{V_{2k-1} \cap C}, \\
    V_{2k+1} &= \overline{V_{2k} \setminus C},
  \end{align*}
  for every $k \ge 0$, where we define $V_{-1}$ to be the ambient space of $C$.
\end{definition}
  \subsection*{Presenting constructible sets as graphs.}
 Throughout, graphs are simple and connected.
  A graph $\cg$ with a distinguished vertex $\xr$ and no cycles is called a \emph{tree}. The vertex $\xr$ is called the \emph{root}. On each edge of $\cg$ we can choose an orientation so that the edge points away from $\xr$. With this orientation, we will view our trees as being directed graphs.

  Let $\ct$ be a tree. If $\xn \to \xn'$ is an edge of $\ct$, then $\xn'$ is called a \emph{child} of $\xn$, and $\xn$ is called the \emph{parent} of $\xn'$.
The vertices with no children are called \emph{leaves}. The root is the only vertex with no parent.
A tree has a natural grading, called $\depth$, given by the path length from the root.

We can represent an irreducible constructible set $C$ as a tree with vertices labeled by varieties. The construction is inductive. If $C$ is closed, we represent it as a tree with a single vertex $\xr$ labeled by $C$. Otherwise, let $V_i$ be the irreducible components of $\overline{C}\setminus C$ and by induction let $\ct_i$ be the tree representation of the constructible set $V_i$. The tree $\ct$ corresponding to $C$ is then constructed as follows: we label the root $\xr$ of $\ct$ by the closure $\overline{C}$ and then attach the root of each $\ct_i$ to $\xr$.

\begin{definition}
  Any labeled tree obtained from a constructible set by the construction above will be called a \emph{constructible tree}.
\end{definition}

From the tree $\ct$ we can recover the corresponding constructible set $C(\ct)$ inductively as follows:
\[
C(\ct) =  V(\xr) \; \backslash \bigcup_{\xn \in \text{children}(\xr)} C(\ct_\xn) ,
\]
where $\ct_\xn \subset \ct$ is the subtree with root $\xn$.
\enlargethispage{\baselineskip}
\begin{lemma}
Let $\ct$ be a constructible tree with two vertices $\xn,\xn' \in \ct$. If $V(\xn) = V(\xn')$, then $\depth(\xn) \equiv \depth(\xn') \Mod{2}$.
\end{lemma}
\begin{proof}
The parity of depth determines whether or not the generic point of $V(\xn)=V(\xn')$ is contained in $C(\ct)$.
\end{proof}
Therefore, to obtain a \emph{constructible graph} from a constructible tree, we may identify vertices having the same label.

\begin{figure}[h!]
\begin{subfigure}{.5\textwidth}
  \centering
  \includegraphics{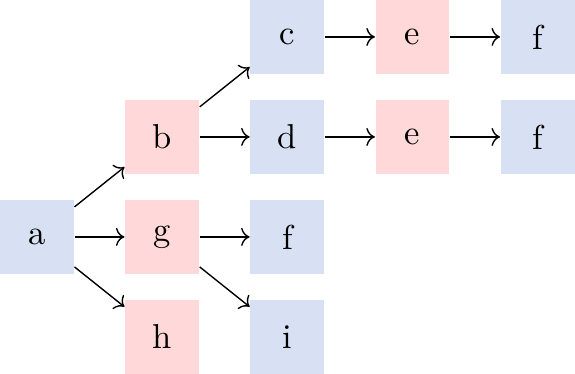}
  \caption{Constructible tree}
\end{subfigure}%
\begin{subfigure}{.5\textwidth}
  \centering
  \includegraphics{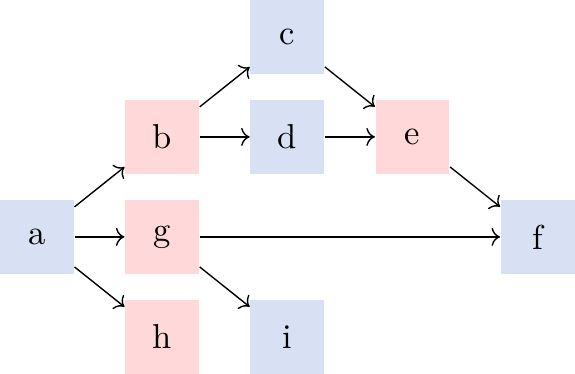}
  \caption{Constructible graph}
\end{subfigure}
\caption{A constructible tree and the resulting constructible graph}
\label{fig:constrGraphExample}
\end{figure}

  \begin{remark}
For our main algorithm, it seems more natural to define and use constructible trees. However, constructible graphs are a more compact representation of a constructible set.
An implementation utilizing the graph structure would like yield savings in time and memory (cf.~\cite{joswig,ganter}).
\end{remark}
\begin{remark}
If $f: \a^n \to \a^m$ is given by monomials, the constructible graph representing $\im(f)$ is a subposet of the face lattice of the Newton polytope of $f$, see \cite{geiger2006toric}.

In general, if $f$ is a toric map,
then $\im(f)$ is a subposet of the face lattice of the polytope of characters defining $f$.
\end{remark}

\subsection{An algorithm for computing images}  \label{sec:algorithm}

Let $X \subset \p^n$ be a variety and $f: X \ratto \p^m$ be a polynomial map $x \mapsto [f_0(x):\dots:f_m(x)]$ in the coordinates of $\p^n$.

\begin{definition}
  The \emph{indeterminacy locus} of $f$ is the subscheme $B$ of $X$ cut out by the ideal $(f_0,\dots,f_m) + I_X$.
  The \emph{image of $f$}, denoted $\im(f)$, is the set $f(X\setminus B)$ where $B$ is the indeterminacy locus of $f$.
  The \emph{image closure of $f$}, denoted $\zim(f)$, is the Zariski closure of the image of $f$.
\end{definition}

\begin{definition}\label{def:exceptional_image}
  An algebraic set $A \subsetneq \zim(f)$ containing the difference $\zim(f) \setminus \im(f)$ will be called an \emph{(image) frame} of $f$, since it covers the boundary of the image.
\end{definition}

  We start by describing a subroutine \tt{Frame} which computes an image frame of $f: X \ratto \p^m$.

The idea is to resolve the map $f$ by the blowup $\blow_f X$ of $X$ along the indeterminacy locus $B$ and compute the image of the exceptional divisor.
\[
  \begin{tikzcd}[every node/.style={fill=orange}]
    \blow_f X \arrow[d,"\pi"'] \arrow[rd,"\tilde f"] & \\
    X \arrow[r,dashed, "f"']                         & \p^m
  \end{tikzcd}
\]
However, if $\dim X$ is strictly greater than the dimension of the image, the images of the exceptional divisors may dominate the image of $f$. To resolve this issue we will cut down the dimension of $X$ by taking an appropriate linear section of $X$.

\begin{definition}
  Let $\Lambda \subset \p^n$ be a linear space and $\delta = \dim X - \dim X\cap \Lambda$. Then $X \cap \Lambda$ is a \emph{codimension $\delta$ linear section of $X$}. If $\zim(f) = \zim(f|_{X\cap \Lambda})$ then the linear section $X \cap \Lambda$ will be called \emph{generic}.
\end{definition}

  Let $\delta := \dim X - \dim \im(f)$ and pick a generic codimension $\delta$ linear section $X':=X\cap \Lambda$ of $X$. Blowing up $X'$ along the indeterminacy locus of $f':=f|_{X'}$ gives a resolution $\tilde f' : \blow_{f'} X' \to \p^m$.  Computing the images of the exceptional divisors via $\tilde f'$ gives a frame of $f'$. This in turn is a frame for $f$. All these statements will be proved in Lemma \ref{lem:approximate}.

\begin{algorithm}
\caption{Frame}\label{alg:partial}
\begin{algorithmic}[1]
\Procedure{Frame}{$f:X \ratto \p^m$}
\State $\delta \gets \dim X - \dim \im(f)$
\If{$\delta > 0$}
	\State $X \gets$ generic codimension  $\delta$ linear section of $X$
    \State $f \gets f|_X$
\EndIf
\State $E \gets$ exceptional divisor of $\blow_f X$
\State $\tilde f \gets $ the resolution $(\blow_f X \to \p^m)$ of $f$
\State \textbf{return} irreducible components of $\tilde{f}(E)$
\EndProcedure
\end{algorithmic}
\end{algorithm}

\begin{lemma}\label{lem:approximate}
Let $X \subset \p^n$ be irreducible.
  Then \tt{Frame}$(f\colon X \ratto \p^m)$ returns the irreducible components of a frame $A$ of $f$.
\end{lemma}
\begin{proof}
  By taking a generic linear section $X'$ of $X$, we make sure $\dim X' = \dim \im(f)$ and $f' := f|_{X'}$ has image closure equal to $\zim(f)$. Then the exceptional divisor $E$ of the blow-up of $X'$ has dimension strictly less than $\zim(f)$. Therefore the image $A$ of $E$ will be \emph{strictly} contained in $\zim(f)$.

  On the other hand, $\im(f') \subset \im(f)$, and $\zim(f')=\zim(f)$.  Therefore  $\zim(f') \setminus \im(f') \supset \zim(f) \setminus \im(f)$ and we need only show the containment $A \supset \zim(f') \setminus \im(f')$.

  The blowup $\blow_{f'}X' \overset{\pi}{\to} X'$ gives a resolution
$\tilde f' : \blow_{f'}X' \to \p^m$ of $f'$. Note that $\im(\tilde{f'}) = \zim(f')$. In particular, for any point $y \in \zim(f') \setminus \im(f')$  we can find a point $x \in \tilde X'$ satisfying $\tilde f'(x)=y$. Then $\pi(x)$ must be in the indeterminacy locus of $f'$. Therefore, $x \in E$ and $y \in A$.
\end{proof}

The idea behind the main algorithm \tt{TotalImage} is to compute successively finer approximations of the image boundary $\zim(f) \setminus \im(f)$. We now give an informal demonstration of how these approximations can be obtained.

Let $Y_0 = \zim(f)$ and $A_0 \supset Y_0 \setminus \im(f)$ be a frame of $f$. Then $Y_0 - A_0 \subset \im(f) \subset Y_0$.  We improve this approximation as follows. Define $X_1 := f\inv(A_0) \subset X$ to be the preimage of $A_0$ and let $Y_1 = \zim(f|_{X_1})$. Note that the image of $f|_{X_1}$ is precisely $A_0 \cap \im(f) \subset Y_1$. In particular,
\begin{align}\label{eq:approx}
  \im(f) &= Y_0  - A_0 + \im(f)\cap A_0 =Y_0  - A_0 + \im(f|_{X_1}).
\end{align}
Let $A_1$ be a frame of $f|_{X_1}$.  This time we have $Y_1 - A_1 \subset \im(f|_{X_1}) \subset Y_1$.
Combining this with Equation (\ref{eq:approx}) gives us
\[
  Y_0 - A_0 + Y_1 - A_1 \subset \im(f) \subset Y_0-A_0+Y_1.
\]
The frames are meant to get strictly smaller in dimension. Therefore, after at most $N := \dim Y_0$ iterations the frame $A_N$ should be empty. This gives
  \[
    Y_0 - A_0 + \dots + Y_{N} - A_N \subset \im(f) \subset Y_0 - A_0 + \dots + Y_{N},
  \]
  which expresses $\im(f)$ exactly when $A_N= \emptyset$.

Note that our algorithm for computing frames uses the irreducibility of the domain in a crucial way. This means that we need to decompose each pullback of a frame into its irreducible components. This causes the algorithm to branch at each step making the construction above harder to visualize. Furthermore, the output will be in the form of a constructible tree. Nevertheless, the nature of the argument remains the same and we prove in Theorem \ref{thm:main} that the resulting constructible tree represents $\im(f)$.

In preparation for the main algorithm, we introduce the following notions:
\begin{enumerate}
  \item Assigning varieties to vertices in the algorithm means creating vertices labeled by these varieties.
  \item Assigning a set of varieties $\{V_1,\dots,V_k\}$ to children($\xn$) means creating $k$ children of $\xn$ which are labeled with $V_1,\dots,V_k$.
  \item A vertex is called \emph{unprocessed} if its label is a subvariety of the \emph{domain} of $f$.
\end{enumerate}
\enlargethispage{\baselineskip}

We are ready to present the main algorithm of this paper, which computes the image of a polynomial map (as we prove in Theorem \ref{thm:main}).

\begin{algorithm}[H]
\label{alg:main}
\caption{TotalImage}
\begin{algorithmic}[1]
\Procedure{TotalImage}{$f:X \ratto \p^m$}
\State $\ct \gets$ the \emph{unprocessed} root vertex $\xr$ labeled by $X$
\ForAll{\emph{unprocessed} vertices $\xn \in \ct$}
  \State children($\xn$) $\gets$ \tt{Frame}$(f|_{V(\xn)})$
  \ForAll{children $\xc$ of $\xn$}
    \State children($\xc$) $\gets$ irreducible components of $\overline{f^{-1}(V(\xc)) \backslash \bs{f}}$
  \EndFor
  \State $V(\xn) \gets \zim(f|_{V(\xn)})$
\EndFor
\State Remove duplicates as in Section \ref{sec:treecleaners} and \textbf{return} $\ct$.
\EndProcedure
\end{algorithmic} %
\end{algorithm} %
 Below, we give a graphical representation of the main loop (line 3) of \texttt{TotalImage}.  The unprocessed nodes are highlighted ($W_1,W_2 \subset X$).
\begin{figure}[h!]
\centering
\includegraphics{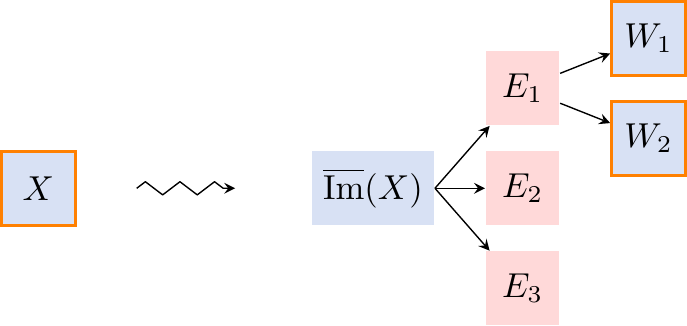} %
\end{figure} %

\subsubsection{Cleaning the tree}
\label{sec:treecleaners}
The tree we construct at the end of the loop in \tt{TotalImage} may contain edges of the form $\xn \to \xn'$ where $V(\xn)=V(\xn')$. This happens when a component of a frame is dominated by the image. Then we can delete $V(\xn)$ as well as all of its descendents and add the descendents of $V(\xn')$ to the parent of $V(\xn)$ (cf. Figure \ref{fig:cleaningChildren}).
\begin{figure}[h!]
\centering
\begin{subfigure}{.5\textwidth}
  \centering
  \includegraphics{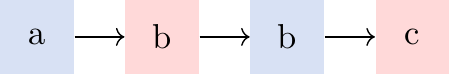} %
\end{subfigure}%
\begin{subfigure}{.5\textwidth}
  \centering
  \includegraphics{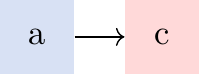} %
\end{subfigure} %
\caption{The tree $\ct$ before and after cleaning} %
\label{fig:cleaningChildren} %
\end{figure} %

  There is another instance of redundancy. It may be that $V(\xn)$ has two children $V(\xn_1')$ and $V(\xn_2')$ with $V(\xn_1') \subset V(\xn_2')$. It is then unnecessary to keep both of these branches of the tree. We will remove $V(\xn_1')$ and all its descendants (cf. Figure \ref{fig:cleaningSiblings}).
\begin{figure}[h!]
\centering
\begin{subfigure}{.5\textwidth}
  \centering
  \includegraphics{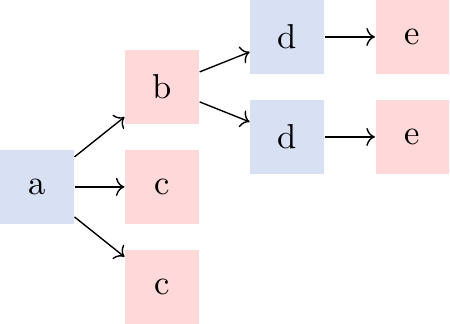} %
\end{subfigure}%
\begin{subfigure}{.5\textwidth}
  \centering
  \includegraphics{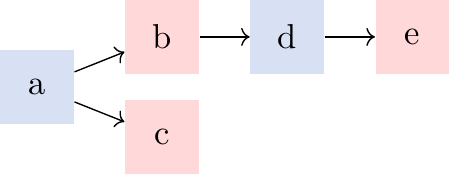} %
\end{subfigure} %
\caption{The tree $\ct$ before and after cleaning} %
\label{fig:cleaningSiblings} %
\end{figure} %

\subsubsection{Justification of the algorithm} Let $f \colon X \ratto \p^m$ be a polynomial map. We start by proving a lemma that our algorithm correctly describes the image.
\begin{lemma}\label{lem:descriptionofimage}
Before we clean the tree, for any point $p$ in (resp.~not in) the image, any longest path starting at the root and going through vertices labeled by varieties containing $p$ goes through an odd (resp.~even) number of vertices.
\end{lemma}
\begin{proof}
The proof is inductive on $\dim\im f$, the case $\dim\im f=0$ being trivial. If $p\not \in \zim f$ we are done, so assume $p\in \zim f$. If $p$ is not in the image of the exceptional divisor then $p\in\im f$ and the claim is true. Otherwise, $p$ belongs to a component $Z$ of the image of the exceptional  divisor. Our algorithm will work on components of $f^{-1}(Z)$  and we conclude by induction. 
\end{proof}
We note that Lemma \ref{lem:descriptionofimage} remains true also after cleaning the tree.
\begin{theorem}\label{thm:main}
  The algorithm \tt{TotalImage} terminates and outputs a constructible tree for the canonical representation of $\im(f)$.
\end{theorem}
\begin{proof}
  The algorithm stops, as the frames (which always appear at odd levels), have dimension strictly smaller than their parents.
  
  Let $V_0 - V_1 + \dots + (-1)^\ell V_\ell$ be the \emph{canonical} representation of $\im(f)$ as in Definition \ref{def:constr}. 
  We prove the following statements by induction on $i$:
  \begin{enumerate}
\item all components of $V_i$ appear at depth $i$ of the tree,
\item all labels at depth $i$ are subvarieties of $V_i$.
\end{enumerate}
Note that for $i$ odd (resp.~even) the components of $V_i$, in the canonical representation, are the largest subvarieties in $V_{i-1}$, with generic points in (resp. not in) $\im f$. The claim is true for $i=0$. Let $Z$ be a component of $V_i$ for $i$ even (resp.~odd). Suppose $Z\subset Y$ for $Y$ a component of $V_{i-1}$. Then $Z$ must be a child of $Y$ by Lemma \ref{lem:descriptionofimage}, which proves the first point.

For the second point consider a label $W$ at depth $i$, a child of $Y$.  By induction $Y$ belongs to a component $\tilde Y$ of $V_{i-1}$. A generic point of $W$ is (resp.~is not) in the image, so $W$ must belong to a component of $V_i$.\qedhere
\end{proof}

\begin{corollary}
 The algorithm \tt{TotalImage} returns a single node if and only if $\im(f) = \zim(f)$.
\end{corollary}
\begin{proof}
  The canonical representation of a set has a single term if and only if the set is closed.
\end{proof}

\subsection{Running time}\label{sub:comparison}
We compared our implementation to \texttt{PolynomialMapImage}.  Now we present timings for comparison.
We used the following examples as benchmarks:
\begin{enumerate}
\item The Whitney Umbrella\label{item:whitney} example from documentation\footnote{\scriptsize \url{https://www.maplesoft.com/support/help/maple/view.aspx?path=RegularChains\%2FConstructibleSetTools\%2FPolynomialMapImage} }
for \tt{PolynomialMapImage}:
  \[
    (x,y)\mapsto (xy,x,y^2).
  \]
\item The homogenization of the map giving the Whitney Umbrella:
  \[
    (x,y,z)\mapsto (xy,xz,y^2).
  \]
\item The map $G$ given by the gradient of $xyz(x+y+z)$:
  \[
    (x,y,z) \mapsto (2xyz+y^2z+yz^2,x^2z+2xyz+xz^2,x^2y+xy^2+2xyz).
  \]
\item The composition $G\circ G$ of the map $G$ from the previous item.
\item The map defined by three random ternary cubics.
\item The map defined by three random ternary sextics.
\item The map defined by three random ternary cubics vanishing on a fixed point.
\item The map defined by three random ternary quadrics vanishing on a fixed point.
\item The Cavender--Farris--Neyman model---see Section \ref{subsec:CFN}.
\item The map defining $\imps(2,2,3)$ (after restricting the domain)---see Section \ref{subsec:IMPS}.
\end{enumerate}
\begin{figure}[h]
\centering
\makebox[0pt]{
  \small
  \begin{tabular}{l*{11}{r}}
  \toprule
  Example & 1 & 2 & 3 & 4 & 5 & 6 & 7 & 8 & 9 & 10\\
  \midrule
  \texttt{TotalImage} & 0  & 0   & 2   & 20  & 0   & 0   & 1   & 0  & 29 & 3 \\
  \texttt{PolynomialMapImage} & 0 & 0 & 22 & -- & -- & -- & -- & 3237 & 31 & 1 \\
  \bottomrule
  \end{tabular} %
} %
\caption{Timings (in seconds) on an Intel Xeon E7-8837 (2.67 GHz) processor. Runtimes of more than 100 hours are designated by ``--''.} %
\label{table:timings} %
\end{figure} %

Let us point out that the two algorithms have outputs of a different nature. As an example we compare our outputs for Item \ref{item:whitney} in our list, the Whitney Umbrella. The image is a closed surface in $\Cc^3$. This is demonstrated by the fact that our output has a single node:
\begin{verbbox}
(2) ideal(y^2*z-x^2).
\end{verbbox}
\begin{figure}[h!]
  \centering
  \theverbbox %
\end{figure} %

\noindent However, \tt{PolynomialMapImage} gives a \emph{triangular decomposition}, representing the same surface as
\begin{verbbox}
[[x^2-y^2*z], [y]], [[x, y], [1]].
\end{verbbox}
\begin{figure}[h!]
  \centering
  \theverbbox %
\end{figure} %
\enlargethispage{\baselineskip}
\vspace{-0.3in}

\section{Site-independent (cyclic) matrix product state}\label{sec:Hack}

Matrix product states ($\mps$) and their more symmetric version---site-independent cyclic matrix product states ($\imps$)--- play an important role in quantum physics and quantum chemistry \cite{Quantum}. They are applied, for instance, to compute the eigenstates of the Schr\"odinger equation. As numerical methods are often involved in their study, the question of the closedness of families of tensors that allow such representations are central and were asked by W.~Hackbusch and L.~Grasedyck.

To answer these questions we present the families of tensors that allow a representation as a matrix product state as orbits under a group action. The equivalence of the classical definition and ours is proved in the Appendix.

We begin by picking a special element in $\imps$ and describe $\imps$ as the orbit of this element with respect to change of coordinates. This allows us to work with an explicit parametrization of $\imps$ and we will show that this parametrization map does not have closed image, proving that $\imps$ is not closed. The element we pick for this purpose is the \emph{iterated matrix multiplication tensor}.

Since what we do here works equally well over $\RR$ or $\Cc$ we use the letter $\kk$ to stand for one of these fields.

\begin{definition}[Iterated matrix multiplication tensor]
For positive integers $a_1,\dots,a_q$ define the tensor $\xM_{a_1,\dots,a_q}\in \kk^{a_1\times a_2}\otimes\kk^{a_2\times a_3}\otimes\cdots\otimes\kk^{a_q\times a_1}$ as
\[
  \xM_{a_1,\dots,a_q}:=\sum_{1\leq i_j\leq a_j}e_{i_1,i_2}\otimes e_{i_2,i_3}\otimes\dots\otimes e_{i_{q-1},i_q}\otimes e_{i_q,i_1},
\]
where $e_{i_j,i_{j+1}}$ are the basis vectors of the space of matricies $\kk^{a_j\times a_{j+1}}$.
\end{definition}

The following statement maybe taken as a working definition of $\imps$ and $\mps$.
The result itself is a generalization of \cite[Proposition 2.0.1]{landsburg2012geometry}.
  \begin{proposition} \label{prop:orbit}
    The sets $\imps$ and $\mps$ may be represented as
    \begin{enumerate}
      \item $\imps(r,k,q)=\{f^{\otimes q}(\xM_{r,\dots,r}) \mid f\in \Hom(\kk^{r\times r},\kk^{k})\}$,
      \item $\mps(\mathfrak{a},\mathfrak{b},q)=\{(f_1\otimes\dots \otimes f_q)(\xM_{a_1,\dots,a_q}) \mid f_i \in \Hom(\kk^{a_i \times a_{i+1}},\kk^{b_i})\}$,\\
      where $\mathfrak{a}=(a_1,\dots,a_q),\mathfrak{b}=(b_1,\dots,b_q)$.
    \end{enumerate}
  \end{proposition}
  \begin{proof}
    See Proposition \ref{prop:orbit-w-proof}
  \end{proof}
\begin{remark}\label{rem:imps_in_mps}
Clearly $\imps(r,k,q)\subset \mps((r,\dots,r),(k,\dots,k),q)$.
\end{remark}
One of the main motivations to start the work on this article was the following question posed by W.\ Hackbusch:
\begin{question}
  Is the set $\imps(r,k,q)$ closed for every $k,r$ and $q$?
\end{question}
\enlargethispage{\baselineskip}
To be more precise, W.~Hackbusch expected a negative answer to the above question and also asked for an explicit tensor $T\in \overline{\imps(r,k,q)}\setminus \imps(r,k,q)$. An analogous question for $\mps$ was asked by L.~Grasedyck in the context of quantum information theory and was completely answered in \cite{landsburg2012geometry}.

It is an easy exercise to show that when $q=2$ both $\imps$ and $\mps$ are closed. Below we will present infinitely many values of $(r,k,q)$ for which $\imps(r,k,q)$ is not closed. In fact, we give an explicit tensor $T$ in $\overline{\imps(r,k,q)}$ such that $T$ is not even in $\mps((r,\dots,r),(k,\dots,k),q)$, let alone in $\imps(r,k,q)$ (see Remark \ref{rem:imps_in_mps}). This demonstrates that $\mps$ is also not closed in these sets of examples, as predicted by Theorem 1.3.2 of \cite{landsburg2012geometry}.

Using Proposition \ref{prop:orbit} we may describe $\imps(r,k,q)$ by the following parametrization map
\[
  \Par :\Hom(\kk^{r\times r},\kk^{k}) \to (\kk^k)^{\otimes q} : M\mapsto (M^{\otimes q})(\xM_{r,\dots,r}).
\]
This puts us exactly within the context of the current article.

  We now show that the image of $\Par$ is not closed by constructing a point in its closure which is demonstrably not hit by $\Par$.
Here we will use the idea of approximating the boundary of the image (cf.~Section \ref{sec:algorithm}).
In general, the approximation is done by blowing up the indeterminacy locus and computing its image. However, individual points in this approximate boundary may be constructed analytically by approaching the indeterminacy locus along a path \[\gamma: (0,1] \to \Hom(\kk^{r\times r},\kk^{k})\] and computing the limit $\lim_{t \to 0} \Par\circ \gamma(t)$.

\begin{theorem}\label{tw:main}
$\imps(2,4,3)$ is not closed. In fact, there exists a curve $c: (0,1] \to \imps(2,4,3)$ for which $\lim_{t \to 0} c(t)$ does not even belong to $\mps((2,2,2),(4,4,4),3)$.
\end{theorem}
\begin{proof}
Let $e_{11},e_{12},e_{21},e_{22}$ be the (standard) basis of $\kk^{2\times 2}$ and $b_1,\dots,b_4$ be the basis of $\kk^4$.
We fix an element $M\in \Hom(\kk^{2\times 2},\kk^4)$ which is defined by
\[
  M(e_{ij}) =
\left\{\begin{array}{ccc}
0  & : & (i,j) \neq (1,2)\hphantom{.} \\
b_2 & : & (i,j) = (1,2)  .
\end{array} \right.
\]

Note that $M$ belongs to the indeterminacy locus of $\Parr$ since $\Parr(M) = M^{\otimes 3}(\xM_{2,2,2})=0$ as can be immediately verified:
\[
  M^{\otimes 3}(\xM_{2,2,2})= \sum_{i_1,i_2,i_3} M(e_{i_1 i_2})\otimes M(e_{i_2 i_3}) \otimes M(e_{i_3 i_1}).
\]
For any term in the summand, the first factor $M(e_{i_1 i_2})$ is non-zero if and only if $i_1=1$ and $i_2=2$. But then the second factor $M(e_{i_2 i_3})$ vanishes.

Let $\fl\in \Hom(\kk^{2\times 2},\kk^4)$ be the flattening isomorphism defined by
\[
  \fl(e_{11})=b_1,\, \fl(e_{12})=b_2,\, \fl(e_{21})=b_3,\, \fl(e_{22})=b_4.
\]

Consider the curves $\gamma(t):=(M+t\cdot \fl)$ and $c(t):=\frac{1}{t^2} \Parr(\gamma(t))$. Let us denote by $e_{11}^\vee,e_{12}^\vee, e_{21}^\vee,e_{22}^\vee$ the dual basis to $e_{11},e_{12},e_{21},e_{22}$. Then we can write
\[
  \gamma(t) = t e_{11}^\vee \otimes b_1 + (1+t) e_{12}^\vee \otimes b_2 + t e_{21}^\vee \otimes b_3 + t e_{22}^\vee \otimes b_4.
\]
From this point onwards we suppress the tensor notation, writing the tensor product as ordinary product, as no confusion is likely. Recall that we have
\begin{align*}
  \xM_{2,2,2} & = e_{11}e_{11}e_{11} + e_{11}e_{12}e_{21} + e_{12}e_{21}e_{11} + e_{12}e_{22}e_{21} \\
                       & + e_{21}e_{11}e_{12} + e_{21}e_{12}e_{22} + e_{22}e_{21}e_{12} + e_{22}e_{22}e_{22}.
\end{align*}
Therefore we can write $\Parr \circ \gamma(t) = \gamma(t)^{\otimes 3}(\xM_{2,2,2})$ as follows:
\enlargethispage{\baselineskip}
\begin{align*}
  \Parr \circ \gamma(t) & = t^3 b_1^{ 3} + t^2 (1+t) b_1 b_2  b_3 + t^2 (1+t) b_2  b_3  b_1 + t^2 (1+t) b_2  b_4  b_3 \\
                       & + t^2(1+t) b_3  b_1  b_2 + t^2 (1+t) b_3  b_2  b_4 + t^2 (1+t) b_4  b_3  b_2 + t^3 b_4^{ 3}.
\end{align*}
It is now clear that $\Parr \circ \gamma(t) \neq 0$ when $t\neq 0$ so that $\Parr([\gamma(t)])$ is well-defined. We then define:
\begin{align*}
  D :={}& \lim_{t \to 0} c(t) \\
  ={}& b_1b_2b_3 + b_3b_1b_2 + b_2b_3b_1 + b_4b_3 b_2 + b_2b_4b_3 + b_3b_2b_4.
\end{align*}

We now prove that $D$ is not in $\mps$. Suppose for contradiction that it were. Then using Proposition \ref{prop:orbit} we can find three linear maps $L_1,L_2,L_3$ in $\Hom(\kk^{2\times 2}, \kk^4)$ such that $D = (L_1 \otimes L_2 \otimes L_3) (\xM_{2,2,2})$. We will now show that each $L_i$ is an isomorphism.

Denoting by $V_i \subset \kk^4$ the image of $L_i$ we have $D \in V_1 \otimes V_2 \otimes V_3$ by design. Contracting the second and third tensors via $V_2^\vee$ and $V_3^\vee$ respectively, the element $D$ may also be viewed as a linear map
\[
  D_{1} : V_2^\vee \otimes V_3^\vee \to V_1.
\]
However, it is clear that the image of $D_{1}$ is $\langle b_1,b_2,b_3,b_4 \rangle$. This forces $V_1 = \kk^4$ which in turn implies $L_1$ is an isomorphism. Similarly, we can show $L_2$ and $L_3$ are isomorphisms.

Therefore, $D$ is isomorphic to $\xM_{2,2,2}$. The multiplication tensor $\xM_{2,2,2}$ is known to have tensor rank 7 \cite{hopcroft1971minimizing, winograd1971multiplication, landsberg2006border, hauenstein2013equations}, but we already have a rank 6 decomposition of $D$, a contradiction.
\end{proof}

We stated Theorem \ref{tw:main} in a way that the proof could be written explicitly. However, with minor modification the proof extends to the case of arbitrary odd $q$.

\begin{theorem} \label{thm:generalized_main}
  $\imps(2,4,q)$ is not closed whenever $q > 1$ is odd.
\end{theorem}

\begin{proof}
  Here we will simply outline the proof in comparison to the proof of Theorem \ref{tw:main}. Take the same $M \in \Hom(\kk^{2 \times 2}, \kk^4)$ and $\gamma(t) = M + t \cdot \fl$. Let $\xM_{2,\dots,2}$ be the iterated matrix product tensor with 2 repeated $q$ times. As before, define the tensor
  \begin{align*}
    D :={}& \lim_{t \to 0} \frac{1}{t^{\lceil \frac{q}{2}\rceil}}\Par \circ \gamma(t).
  \end{align*}
  It will be sufficient to show $D$ is not in $\mps$. However, the contraction maps $D_i$ induced by $D$ all have surjective images when $q$ is odd. Therefore, $D$ is in $\mps$ if and only if $D$ is isomorphic to $\xM_{2,\dots,2}$. But $D$ has rank at most $2q$ whereas $\xM_{2,\dots,2}$ has rank at least $2^{q-1}$ \cite[Proposition 20]{buhrman2016nondeterministic}.
\end{proof}

\section{Examples}\label{sec:examples}

\subsection{The Cavender--Farris--Neyman model}\label{subsec:CFN}
The examples below are inspired by statistics. They represent a type of group-based model, which is a special Markov process on trees \cite{sturmfels2005toric, michalek2011}.

The map $\varphi: \Cc^6 \to \Cc^8$ defined below represents the Cavender--Farris--Neyman model (also known as the 2-state Jukes--Cantor model) for the tripod \cite{buczynska07}:
\begin{align*}
(a,b,c,d,e,f)\mapsto &(ace + bdf, acf + bde,
    ade + bcf,
    bce + adf,\\
    & \phantom{(}bde + acf,
    bcf + ade,
    adf + bce,
    bdf + ace).
\end{align*}
There are four obvious independent linear phylogenetic invariants---linear polynomials vanishing on the image. In fact, it is well known that the closure of the image is a four dimensional linear space \cite{buczynska07, sturmfels2005toric}. Our algorithm \tt{TotalImage}$(\varphi)$ returns the output in Figure \ref{fig:graphmodelOutput}. There are four linear spaces of dimension three, whose generic points do not belong to the image closure $\zim(\varphi)$. There are six distinct planes which are added back in. This provides a complete description of the image in statistically meaningful coordinates (without applying the discrete Fourier transform).
\small\begin{verbbox}
   (4) ideal(y3-y6,y2-y5,y1-y4,y0-y7)
 - (3) |====ideal(y1-y4,y0-y2+y3-y4,-y2+y5,-y3+y6,-y2+y3-y4+y7)
 + (2) |    |====ideal(y5-y6,y4-y7,y3-y6,y2-y6,y1-y7,y0-y7)
 + (2) |    |====ideal(y6+y7,y4+y5,y3+y7,y2-y5,y1+y5,y0-y7)
 + (2) |    |====ideal(y5-y7,y4-y6,y3-y6,y2-y7,y1-y6,y0-y7)
 - (3) |====ideal(y1-y4,y0-y2-y3+y4,-y2+y5,-y3+y6,-y2-y3+y4+y7)
 + (2) |    |====ideal(y5+y6,y4+y7,y3-y6,y2+y6,y1+y7,y0-y7)
 + (2) |    |====ideal(y6-y7,y4-y5,y3-y7,y2-y5,y1-y5,y0-y7)
 + (2) |    |====ideal(y5-y7,y4-y6,y3-y6,y2-y7,y1-y6,y0-y7)
 - (3) |====ideal(y1-y4,y0+y2-y3-y4,-y2+y5,-y3+y6,y2-y3-y4+y7)
 + (2) |    |====ideal(y5-y6,y4-y7,y3-y6,y2-y6,y1-y7,y0-y7)
 + (2) |    |====ideal(y6-y7,y4-y5,y3-y7,y2-y5,y1-y5,y0-y7)
 + (2) |    |====ideal(y5+y7,y4+y6,y3-y6,y2+y7,y1+y6,y0-y7)
 - (3) |====ideal(y1-y4,y0+y2+y3+y4,-y2+y5,-y3+y6,y2+y3+y4+y7)
 + (2) |    |====ideal(y5+y6,y4+y7,y3-y6,y2+y6,y1+y7,y0-y7)
 + (2) |    |====ideal(y6+y7,y4+y5,y3+y7,y2-y5,y1+y5,y0-y7)
 + (2) |    |====ideal(y5+y7,y4+y6,y3-y6,y2+y7,y1+y6,y0-y7)
\end{verbbox}
\begin{figure}[H]
  \centering
  \theverbbox %
\caption{\texttt{TotalImage} output for $\varphi$} %
\label{fig:graphmodelOutput} %
\end{figure} %
\normalsize
As the pairs of parameters $(a,b), (c,d)$ and $(e,f)$ represent probabilities, we may add conditions $a+b=1$, $c+d=1$ and $e+f=1$. It is known \cite{filo} that this adds exactly one additional linear constraint to the closure of the image: namely that all coordinates sum up to a constant. Further, from this three-dimensional affine space we have to subtract three two-dimensional subspaces and add to each two-dimensional subspace a line and a point. 

\subsection{The locus $\imps (2,2,3)$ is closed}\label{subsec:IMPS}
Here we describe the map $\Parrr$ explicitly in coordinates.
An element in the domain of $\Parrr$ is a pair of $2 \times 2$ matricies $(M,L)$ which we write as
  \begin{align*}
    \left(\begin{bmatrix}
        A &B \\
        C &D
    \end{bmatrix},
    \begin{bmatrix}
        a &b \\
        c &d
    \end{bmatrix}\right).
  \end{align*}
  The map $\Par$ takes this to the coordinate vector
  \[\small
  \begin{bmatrix}
  \tr(MMM)\\
  \tr(LMM)\\
  \tr(MLM)\\
  \tr(MML)\\
  \tr(LLM)\\
  \tr(LML)\\
  \tr(MLL)\\
  \tr(LLL)
  \end{bmatrix}
  =
    \begin{bmatrix}
    A^3+3ABC+3BCD+D^3 \\
    A^2a+ABc+aBC+BcD+AbC+BCd+bCD+D^2d \\
    A^2a+AbC+ABc+BCd+aBC+bCD+BcD+D^2d \\
    A^2a+aBC+AbC+bCD+ABc+BcD+BCd+D^2d \\
    Aa^2+Abc+aBc+Bcd+abC+bCd+bcD+Dd^2 \\
    Aa^2+aBc+abC+bcD+Abc+Bcd+bCd+Dd^2 \\
    Aa^2+abC+Abc+bCd+aBc+bcD+cdB+Dd^2 \\
	    a^3+3abc+3bcd+d^3
    \end{bmatrix}.
  \]
  There are 4 linear relations among these polynomials which implies that the image lies in a four-dimensional subspace $U$ of $\Cc^8$. In fact, $U=S^3(\Cc^2)$ is the subspace of symmetric tensors.
  We proved that the image is closed and equals $U$ using \tt{TotalImage} in the following way. We restrict the map to pairs of matrices $(M,L)$ where $M$ is diagonal and $L$ has its
  non-diagonal entries equal. Then \tt{TotalImage} can compute that the image, even restricted to this smaller domain, is exactly $U$. In this example, one could also conclude purely theoretically that the image is closed, as the space of symmetric tensors $U$ has only three $GL(2)$ orbits.

\appendix

\section{Matrix product states}
\label{appendix:mps}

We recall here two representations of tensors that are inspired from physics \cite{perez2007matrix}.
Recall we use $\kk$ to stand for $\Cc$ or $\RR$.

For any $a \in \Zz_{>0}$ the vector space $\kk^a$ comes with the standard basis $e_1,\dots,e_a$. Therefore, a tensor $T\in \kk^{a_1}\times\dots\times\kk^{a_q}$ may be represented as
\[
  T=\sum_{1\leq i_j\leq a_j} \lambda_{i_1,\dots,i_q} e_{i_1}\otimes\dots\otimes e_{i_q},
\]
which is also written
\[
  T[i_1,\dots,i_q]=\lambda_{i_1,\dots,i_q}.
\]
\begin{definition}[Site-independent (cyclic) matrix product state]\label{def:sym}
Fix integers $r>0$, $k>0$, $q>1$ and matrices $M_i\in \kk^{r\times r}$ for $i=1,\dots,k$. Let $T\in (\kk^k)^{\otimes q}$ be a tensor given by
\[
  T[i_1,\dots,i_q]:=\tr(M_{i_1}M_{i_2}\cdots M_{i_q}).
\]
The set of all tensors that allow such a representation will be denoted by $\imps(r,k,q)\subset (\kk^k)^{\otimes q}$.
\end{definition}
\begin{example}\label{ex:przyk1}\label{ex:IMPSq=2} 
Let us consider the case of matrices ($q=2$). Here elements of  $\imps(r,k,2)$ can be viewed as matrices $M$ such that $M[i_1,i_2]=\tr(M_{i_1}M_{i_2})$. This is equivalent to a factorization of $M=A\cdot A^t$ for some matrix $A\in \Hom(\kk^r,\kk^{k^2})$. In particular, $M\in \imps(r,k,2)$ if and only if $M$ is symmetric and has rank at most $k^2$. It follows that $\imps(r,k,2)$ is closed.
\end{example}
When $q=2$ the tensor $T$ corresponds to a symmetric matrix. However, for $q >2$ the tensor $T$ will not be a symmetric tensor in general, though the identity $T[i_1,\dots,i_q]=T[i_q,i_1,\dots,i_{q-1}]$ continues to hold. In other words, the tensor has \emph{cyclic} symmetries with respect to the order of the product of the matrices.

Definition \ref{def:sym} can be regarded as a symmetrization of the following definition of a cyclic matrix product state, where the underlying graph for the tensor network is a cycle.

Fix an integer $q > 1$ and tuples of positive integers $\xa = (a_1,\dots,a_q)$, $\xb=(b_1,\dots,b_q)$. We set $a_{q+1}=a_1$. Then the locus $\mps(\xa,\xb,q) \subset \kk^{b_1}\otimes \dots \otimes \kk^{b_q}$ is given by the following definition.

\begin{definition}[Cyclic matrix product state]\label{def:MPS}
A tensor $T \in \kk^{b_1}\otimes \dots \otimes \kk^{b_q}$ is in $\mps(\xa,\xb,q)$ if there exist matrices
\[
  M_{i,j}\in \Hom_{\kk}(\kk^{a_j},\kk^{a_{j+1}}),\quad j=1,\dots,q,\, i = 1,\dots,b_j,
\]
such that
\[
  T[i_1,\dots,i_q]:=\tr(M_{i_1,1}M_{i_2,2}\cdots M_{i_q,q}).
\]
\end{definition}

\begin{example} \label{rem:MPSq=2}The situation for $q=2$ is analogous to Example \ref{ex:przyk1}. In this case, we have $M\in \mps((a_1,a_2),(b_1,b_2),2)$ if and only if $M=AB$ where $A\in \Hom(\kk^b_1,\kk^{a_1a_2})$ and $B\in \Hom(\kk^{a_1a_2},\kk^{b_2})$. This can happen if and only if the rank of the matrix $M$ is at most $a_1a_2$. Therefore, $\mps(\mathfrak{a},\mathfrak{b},2)$ is always closed.
\end{example}

  \begin{proposition} \label{prop:orbit-w-proof}
    The sets $\imps$ and $\mps$ may be represented as
    \begin{enumerate}
      \item $\imps(r,k,q)=\{f^{\otimes q}(\xM_{r,\dots,r}) \mid f\in \Hom(\kk^{r\times r},\kk^{k})\}$.
      \item $\mps(\mathfrak{a},\mathfrak{b},q)=\{(f_1\otimes\dots \otimes f_q)(\xM_{a_1,\dots,a_q}) \mid f_i \in \Hom(\kk^{a_i \times a_{i+1}},\kk^{b_i})\}$.
    \end{enumerate}
  \end{proposition}
   \begin{proof}
   The proofs of both statements are similar. We prove the first one, as it is more important for this paper. We will be interpreting elements of $\Hom(\kk^{r\times r},\kk^k)$ as $r^2\times k$ matrices. First we note that there is a natural bijection $\varphi$ between $k$-tuples of $r\times r$ matrices $\mathcal{M}:=(A_1,\dots,A_k)$ and matrices $\varphi(\mathcal{M})\in \Hom(\kk^{r\times r},\kk^k)$. For $1\leq i\leq k$ the $i$-th column of $\varphi(\mathcal{M})$ is the representation of $A_i$ as a vector of length $r^2$.

   Write $M_i=\sum_{p,q=1}^r a_{i,p,q} e_{p,q}$, where $e_{p,q}$ is the matrix with a 1 in its $(p,q)$-th entry and zeros everywhere else. Note that $\varphi(\mathcal{M})(e_{p,q})=a_{i,p,q}$.

   We prove the claim by showing that the tensor $T\in \imps(r,k,q)$ associated to $\mathcal{M}$ equals $\varphi(\mathcal{M})(\xM_{r,\dots,r})$. Indeed, we have

   \begin{align*}
     T &=\sum_{1\leq i_j\leq k} \tr(M_{i_1}\cdots M_{i_{q}}) e_{i_1}\otimes\dots\otimes e_{i_q}  \\
       &= \sum_{1\leq i_j\leq k}\left(\sum_{1\leq p_j\leq r} a_{i_1,p_1,p_2}a_{i_2,p_2,p_3}\cdots a_{i_{q-1},p_{q-1},p_q}a_{i_q,p_q,p_1}\right)e_{i_1}\otimes\dots\otimes e_{i_q},
   \end{align*}
   where in all sums $1\leq j\leq q$. We can simplify further:
   \begin{align*}
     T&= \sum_{\substack{1\leq i_j\leq k \\ 1\leq p_j\leq r}}(a_{i_1,p_1,p_2}e_{i_1})\otimes\cdots\otimes (a_{i_q,p_q,p_1}e_{i_q}) \\
     &=\sum_{1\leq p_j\leq r}\left(\sum_{1\leq i_1\leq k}a_{i_1,p_1,p_2}e_{i_1}\right)\otimes\cdots\otimes \left(\sum_{1\leq i_q\leq k}a_{i_q,p_q,p_1}e_{i_q}\right) \\
   &=\sum_{1\leq p_j\leq r}\varphi(\mathcal{M})(e_{p_1,p_2})\otimes\cdots\otimes\varphi(\mathcal{M})(e_{p_q,p_1}) \\
     &=\varphi(\mathcal{M})^{\otimes q}(\xM_{r,\dots,r}).\qedhere
   \end{align*}
   \end{proof}


\newcommand{\etalchar}[1]{$^{#1}$}

\end{document}